\documentclass{amsart}
\usepackage{amsmath}
\usepackage{amsfonts}
\usepackage{amssymb}
\usepackage{color}
\usepackage{graphics,graphicx}
\usepackage{hyperref}
\usepackage{amssymb,latexsym,amsthm,xypic}

\newtheorem{thm}{Theorem}[section]
\newtheorem{que}{Question}

\newtheorem{lem}[thm]{Lemma}

\theoremstyle{definition}
\newtheorem{defn}{Definition}[section]
\theoremstyle{remark}
\newtheorem{rem}{Remark}[section]
\newtheorem{exm}{Example}[section]

\newcommand{\be}{\begin{equation}}
\newcommand{\ee}{\end{equation}}
\newcommand{\bea}{\begin{eqnarray}}
\newcommand{\eea}{\end{eqnarray}}
\newcommand{\ben}{\begin{eqnarray*}}
\newcommand{\een}{\end{eqnarray*}}
\newcommand{\bt}{\begin{split}}
\newcommand{\et}{\end{split}}
\newcommand{\bet}{\begin{equation}}
\newcommand{\mc}{\mathbb{C}}
\newcommand{\mr}{\mathbb{R}}

\newcommand{\ra}{\rightarrow}
\newcommand{\bi}{\begin{itemize}}
\newcommand{\ei}{\end{itemize}}

\begin{document}

\title[Exposing boundary points]{Exposing boundary points of strongly pseudoconvex subvarieties in complex spaces}
\date{}
\author[F. Deng, J. E. Forn{\ae}ss, E. F. Wold]{F. Deng, J. E. Forn{\ae}ss, and E. F. Wold}
\address{F. Deng: Matematisk Institutt, Universitetet i Oslo, Postboks 1053 Blindern, 0316 Oslo, Norway;
 and School of Mathematical Sciences, University of Chinese Academy of Sciences, Beijing 100049, China}
\email{fushengd@math.uio.no}
\address{J. E. Forn{\ae}ss: Department for Mathematical Sciences, Norwegian University of Science
and Technology, Trondheim, Norway}
\email{john.fornass@math.ntnu.no}
\address{E. F. Wold: Matematisk Institutt, Universitetet i Oslo, Postboks 1053 Blindern, 0316 Oslo, Norway}
\email{erlendfw@math.uio.no}

\begin{abstract}
We prove that all locally exposable points in a Stein compact
in a complex space can be exposed along a given curve to a given real hypersurface.
Moreover, the exposing map for a boundary point can be sufficiently close to the identity map
outside any fixed neighborhood of the point.
We also prove a parametric version of this result
for bounded strongly pseudoconvex domains in $\mc^n$.
For a bounded strongly pseudoconvex domain in $\mc^n$ and a given boundary point of it,
we prove that there is a global coordinate change on the closure of the domain
which is arbitrarily close to the identity map with respect to the $C^1$-norm
and maps the boundary point to a strongly convex boundary point.
\end{abstract}

\maketitle

\section{Introduction}\label{sec:introduction}
Let $X$ be a complex space.
We assume throughout this paper that
all complex spaces are reduced, irreducible, and paracompact.
Let $X_{\mathrm{sing}}$ be the set of singular points of $X$
and let $X_{\mathrm{reg}}=X\backslash X_{\mathrm{sing}}$
be the set of smooth points of $X$.

\begin{defn}
Let $X$ be a complex space, let $K\subset X$ be a compact set, and let $\zeta\in K$
be a point in $X_{\mathrm{reg}}$.   We will say that $\zeta$ is \emph{locally exposable} if
there exists an open set $U\subset X$ containing $\zeta$ and $\rho$ a $C^2$-smooth strictly plurisubharmonic
function on $U$ such that
\begin{itemize}
\item[(i)] $\rho(\zeta)=0$ and $\mathrm{d}\rho(\zeta)\neq 0$, and
\item[(ii)] $\rho<0$ on $(K\cap U)\setminus\{\zeta\}$.
\end{itemize}
\end{defn}

Our main concern in this paper is to show that locally exposable points are globally
exposable (see Definition \ref{exposable}).

%

The first result of the present paper is the following.

\begin{thm}\label{thm:exposing manifolds}
Let $X$ be a complex space, let $K\subset X$ be a Stein compact,
and let $\zeta\in K\cap X_{\mathrm{reg}}$ be locally exposable.
Let $H\subset X\backslash (K\cup X_{\mathrm{sing}})$ be a germ of a closed real $C^2$-smooth hypersurface  in $X$.
Let $\gamma:[0,1]\rightarrow X_{\mathrm{reg}}$
be a smoothly embedded curve in $X$ with $\gamma(0)=\zeta$,
$\gamma(1)\in H$,
and $\gamma(t)\in X\backslash (K\cup H)$ for $t\in (0, 1)$. \

Then for any (small) neighborhood $V$ of $\gamma$, any neighborhood $V'\subset V$ of $\zeta$, and any $\epsilon>0$,
there exist an open neighborhood $U$ of $K$,
and a biholomorphic map $f:U\rightarrow f(U)\subset X$
such that the following holds:

\begin{itemize}
\item[(a)] $f(V')\subset V$ and $f(\zeta)=\gamma(1)$,
\item[(b)] $\mathrm{dist}(f(z),z)<\epsilon$ for $z\in K\backslash V'$, and
\item[(c)] $f(K)\cap H=\{\gamma(1)\}$.
\end{itemize}
where $dist(\cdot, \cdot)$ is a fixed distance on $X$.
In the case $X=\mc^n$ and $K\subset\mc^n$ is polynomially convex,
$f$ can be taken to be a holomorphic automorphism of $\mc^n$.
\end{thm}
By saying that $K$ is a Stein compact we mean that $K$ has a Stein neighbourhood
basis in which $K$ is holomorphically convex.  \

\begin{defn}\label{exposable}
The map $f$ will be said to \emph{expose} the point $\zeta$ with respect to $H$, and
$f(\zeta)$ is said to be exposed.  If $K\subset B_r(0)\subset\mathbb C^n$ and $H=\{z\in\mathbb C^n:\|z\|=R\}$ we
will say that $f(\zeta)$ is globally exposed.
\end{defn}

From the proof, we will see that Theorem \ref{thm:exposing manifolds} can be generalized in various directions as follows:
\bi
\item one can easily expose finitely many points simultaneously,
\item the last statement can be generalized as: if $X$ is a Stein manifold  and has density property (for definition see \cite{Varolin01})
      and $K$ is $\mathcal O(X)$-convex, then $f$ can be taken in $\mathrm{Aut_{hol}}X$,
\item if $X$ is a 1-convex space,
      the same statement still holds if we assume that $\zeta$ is outside the exceptional set.
      This can be proved by Remmert's reduction and interpolation for the exposing maps constructed in Theorem \ref{thm:exposing manifolds}.
      (Exposing boundary points in this setting was proposed by Franc Forstneri\v{c} in \cite{Forstneric13}.)
\ei

The importance of Theorem \ref{thm:exposing manifolds} in the case that $K$ is the closure of a domain with a strictly
pseudoconvex boundary point $\zeta$, is that it tells us that $K^\circ$ resembles a strictly convex domain in a concrete
geometric sense.  Convex domains are exceptionally well behaved from a complex analysis point of view.  As
a comparison, the Levi problem was to show that such a domain $K^\circ$ resembles a strictly convex domain
in a much weaker function theoretic sense.

%
%

Our next result concerns in which ways a locally exposable point can be locally exposed with global maps.
It is quite simple using Anders\'{e}n-Lempert theory to show that a locally defined exposing map in $\mathbb C^n$ can be
approximated by holomorphic automorphisms, but as a consequence one looses completely control
of the behaviour on most of $K$.

\begin{thm}\label{thm:local expose}
Let $D$ be a bounded strongly pseudoconvex domain in $\mathbb{C}^n$ with $C^2$-smooth boundary.
Then for any $\zeta\in\partial D$ and any $\epsilon >0$,
there is an injective holomorphic map $F:\overline D\rightarrow \mathbb{C}^n$
such that $F(\zeta)=\zeta$ is a strictly convex boundary point of $F(D)$,
and $||F-Id||_{C^1(\overline D)}<\epsilon$.
\end{thm}
In  \S \ref{sec:local transf} we will construct a strongly pseudoconvex domain
such that the map $F$ in Theorem \ref{thm:local expose}
can not be taken in $\mathrm{Aut_{hol}}(\mc^n)$ for some boundary points. \\

Furthermore we are interested in parametric version of the previous two theorems.

\begin{thm}\label{thm:exposing family manifolds}
Let $D\subset\subset\mc^n$ be a strongly pseudoconvex domain with smooth boundary.
Let $H\subset\mc^n$ be a smooth closed real hypersurface with $\overline D\cap H=\emptyset$.
Let $\gamma:\partial D\times [0,1]\rightarrow \partial D\times \mc^n$ be a fiber-preserving
continuous map such that for all $\zeta\in\partial D$:
\begin{itemize}
\item[1)] $\gamma_\zeta=\gamma(\zeta, \cdot):I\rightarrow \mc^n$ is a smooth embedding;
\item[2)] $\gamma_\zeta(0)=\zeta, \gamma_\zeta(1)\in H$ and $\gamma_\zeta|_{(0,1)}\subset \mc^n\backslash (\overline D\cup H)$.
\end{itemize}
Then for any neighborhoods $V$ of $\gamma(\partial D\times [0,1])$
and any sufficiently small neighborhood $V'\subset V$ of $\{(\zeta,\zeta);\zeta\in\partial D\}$ in $\partial D\times \mc^n$
and any $\epsilon >0$,
there exists a smooth fiber-preserving map $f:\partial D\times \overline D\rightarrow \partial D\times \mc^n$
such that the following holds for each $\zeta\in\partial D$:

\begin{itemize}
\item[(a)] $f_\zeta:=f(\zeta, \cdot):\overline D\rightarrow \mc^n$ is injective and holomorphic,
\item[(b)] $f_\zeta(V'_\zeta)\subset V_\zeta$ and $f_\zeta(\zeta,\zeta)=\gamma_\zeta(1)$,
\item[(c)] $||f_\zeta(z)- z||<\epsilon$ for $z\in\overline D\backslash V'_\zeta$,
\item[(d)] $f_\zeta(\overline D)\cap H=\{\gamma_\zeta(1)\}$,
where $V_\zeta:=V\cap(\{\zeta\}\times\mc^n)$.
\end{itemize}
If in addition $\overline D$ is polynomially convex,
we can take $f$ a smooth map from $\partial D\times\mc^n$ to itself
such that $f|_{\mc^n_\zeta}\in Aut(\mc^n)$ for all $\zeta\in\partial D$.
\end{thm}

\begin{thm}\label{thm:local expose family}
Let $D\subset\subset \mc^n$ be a strongly pseudoconvex domain with smooth boundary.
For any $\epsilon>0$,
there is a continuous map $f:\partial D\times\overline D\rightarrow\mc^n$
such that for all $\zeta\in\partial D$ the following hold:
\begin{itemize}
\item[1)] $f_\zeta=f(\zeta,\cdot):\overline D\rightarrow\mc^n$ is a holomorphic injective map,
\item[2)] $f_\zeta(\zeta)=\zeta$ is a strictly convex boundary point of $f_\zeta(D)$,
\item[3)] $\partial f_\zeta(D)$ and $\partial D$ are tangential at $\zeta$, and
\item[4)] $||f_\zeta-ID||_{C^1(\overline D)}<\epsilon$.
\end{itemize}
\end{thm}

For a further discussion of parametric exposing of points, see Section \ref{sec:parametric exposing}.

\section{Transforming a boundary point to a strongly convex one }\label{sec:local transf}
In this section,
we consider transforming a boundary point
of a strongly pseudoconvex domain to a strictly convex one,
with certain control of the behavior of the involved transformation.
The aim is to prove Theorem \ref{thm:local expose}
and Theorem \ref{thm:local expose family}.
We also construct a strongly pseudoconvex domain in $\mc^2$
for which the transformation furnished by Theorem \ref{thm:local expose}
can not be taken in $Aut(\mc^2)$.

We need the following lemmas.
The first lemma is to prove the existence of
peak functions with certain estimates.
The key tool is the existence of embedding of strongly
pseudoconvex domains into strongly convex domains
with certain boundary conditions,
established by the second author in \cite{Fornaess76}.

\begin{lem}\label{lem:estimate peak ftn}
Let $D$ be a bounded strictly pseudoconvex domain
in $\mathbb C^n$ with $C^2$-smooth boundary.
Then for any $\zeta\in\partial D$,
there is a holomorphic function $f$ defined on some neighborhood of $\overline D$
which satisfies the following two conditions:\\
1) $f(\zeta)=1$, and $|f(z)|<1$ for $z\in\overline D\backslash\{\zeta\}$;\\
2) the estimate
  $$|f(z)|\leq e^{-c||z-\zeta||^2}$$
  holds on $\overline D$ for some constant $c>0$.
\end{lem}
\begin{proof}
For the case that $D$ is the ball
$B:=\{z\in\mathbb C^n; |z_1+r|^2+|z_2|^2+\cdots+|z_n|^2<r^2\}$ and $\zeta=0$,
a simple calculation shows that $f(z):=e^{z_1}$ satisfies the conditions.
By the result in \cite{Fornaess76} mentioned above,
the general case can be reduced to the case that $D=B$.
\end{proof}

For the proof of Theorem \ref{thm:local expose family},
we need a parametric version of Lemma \ref{lem:estimate peak ftn}.

\begin{lem}\label{lem:estimate peak ftn family}
Let $D\subset\mc^n$ be a bounded strongly pseudoconvex domain with smooth boundary.
Then there is a smooth map $p:\partial D\times\overline D\rightarrow \mc$ such that:\\
1) For each $\zeta\in\partial D$,
  $p_\zeta(\cdot):=p(\zeta, \cdot)$ is a holomorphic function on $\overline D$;\\
2) $p_\zeta(\zeta)=1$ and $|p_\zeta(z)|<1$ for $z\in\overline D\backslash\{\zeta\}$ for all $\zeta\in\partial D$;\\
3) there exists a constant $c$ such that $|p_\zeta(z)|\leq e^{-c||z-\zeta||^2}$ for $(\zeta,z)\in\partial D\times\overline D$.
\end{lem}
\begin{proof}
By a result in \cite{Fornaess76},
there is a proper holomorphic embedding $\sigma$
from some neighborhood of $\overline D$ into some neighborhood
of a bounded strongly convex domain $W\subset\mc^N$ for some $N$
such that $\sigma(\partial D)\subset \partial W$ and
$\sigma(\overline D)$ and $\partial W$ intersect transversely.  Let $\rho$ be
a defining function for $W$, and set $g_\xi(w)=e^{\partial\rho(\xi)(w-\xi)}$.
Then $g$ satisfies 1-3 for $D$ replaced by $W$.  So we may set $p(\zeta,z)=g_{\sigma(\zeta)}(\sigma(z))$.
\end{proof}

The last lemma we need is the following

\begin{lem}\label{lem:diff top}
Let $D$ be a bounded domain in $\mathbb C^n$ with $C^2$-smooth boundary.
Let $F$ be a diffeomorphism from some neighborhood
of $\overline D$ to its image in $\mathbb C^n$.
Assume that $\{F_j\} (j\geq 1)$ is a sequence of smooth maps
defined on some neighborhood of $\overline D$
such that $||F_j-F||$ converges uniformly on $\overline D$ to 0 in $C^1$-norm.
Then $F_j$ is injective on $\overline D$ for $j$ large enough.
\end{lem}

\begin{proof}
We prove by contradiction.
If it were not the case,
then, without loss of generality,
we can assume that for each $j$ there exist $a_j$ and $b_j$ in $\overline D$ with $a_j\neq b_j$
such that $f_j(a_j)=f_j(b_j)$.
We may assume $a_j\rightarrow a$ and $b_j\rightarrow b$ as $j\rightarrow\infty$.
Then it is necessary that $a=b$ since $F$ is injective.
We have $a\in D$ or $a\in\partial D$.
Since $\partial D$ is $C^2$,
taking a $C^2$-smooth coordinate change near $a$ if necessary,
we may assume that there is a neighborhood $U$ of $a$ such that $U\cap \overline D$ is convex.
Let $\gamma_j$ be the line segment in $\mathbb C^n$ given by $t\rightarrow (1-t)a_j+tb_j,\ t\in [0,1]$,
then $\gamma_j\subset \overline D$ for $j$ large enough.
We have
\begin{equation}
\begin{split}
F_j(b_j)-F_j(a_j)
&=\int_{0}^1\frac{dF_j(\gamma_j(t))}{dt}dt\\
&=(\int_{0}^1dF_j(\gamma_j(t))dt)(b_j-a_j),
\end{split}
\end{equation}
which can not be 0 for $j$ large enough since $\int_{0}^1dF_j(\gamma_j(t))dt\rightarrow dF(a)$
and $dF(a)$ is nonsingular. Contradiction.
\end{proof}

\emph{Proof of Theorem \ref{thm:local expose}:}
We can assume that $\zeta=0$ and the local defining function
$\rho(z)$ of $D$ near $0$ can be expanded as
$$\rho(z)=Re(z_n+Q(z))+\sum_{i,j}a_{ij}z_i\bar z_j+\cdots,$$
where $Q(z)=\sum_{i,j}q_{ij}z_iz_j$ is a symmetric quadratic form.
Let $f$ be a holomorphic function defined in some neighborhood of $\overline D$
that is furnished by Lemma \ref{lem:estimate peak ftn}.
We will show that there are positive integers $M$ and $N_j, j=1, \cdots, M$ such that
the transformation $F(z)=(\tilde z_1, \cdots, \tilde z_n)$ given by
\begin{equation}\label{eqn:map from sum}
\begin{cases}
\tilde z_k=z_k,\ \text{for}\ k=1, \cdots, n-1, \\
\tilde z_n=z_n+\sum_{j=1}^M\frac{1}{M}Qf^{N_j}
\end{cases}
\end{equation}
is that expected by this theorem.

First note that
$$||Q'f^N||\lesssim ||z||e^{-Nc ||z||^2},$$
and hence goes to zero uniformly on $\overline D$
as $N\rightarrow \infty$.
Note also that
\begin{equation}\label{eqn:derivative estimate}
||QNf'f^{N-1}||\lesssim N||z||^2e^{-Nc ||z||^2}.
\end{equation}
Let $\sigma$ be the function on $(0,+\infty)$ given by $x\mapsto xe^{-c x}$.
It is clear that $\sigma(0)=0$ and $\lim_{x\rightarrow +\infty}\sigma(x)=0$,
and $\sigma$ attains its maximum $\frac{1}{c e}$ at $x=1/c$.
So the right hand side of \eqref{eqn:derivative estimate} attains its maximum
$\frac{1}{c e}$ at $||z||^2=\frac{1}{Nc}$.
For any $\epsilon>0$ and a sufficiently large integer $M$ with $\frac{1}{Mec}<\epsilon/2$,
by the above discussion,
we can find $r_1>r_2>\cdots>r_M>0$ and positive numbers $N_1, \cdots, N_M$ such that
$||(Qf^{N_j})'||<\frac{\epsilon}{2M}$ on $\{z\in\overline D; ||z||\geq r_j\ \text{or}\ ||z||\leq r_{j+1}\}$.
Define $\varphi_M=\sum_{j=1}^M\frac{1}{M}Qf^{N_j}$,
then $\varphi_M$ is a holomorphic function defined on some neighborhood of $\overline D$,
and $||\varphi'_M||<\epsilon$ on $\overline D$.
By taking $N_j$ large enough,
we can also require that $|\varphi_M(z)|<\epsilon$ for $z\in\overline D$.
Let $F_M$ be the map defined as in \eqref{eqn:map from sum}.
By Lemma \ref{lem:diff top},
$F_M$ is injective on $\overline D$ for $M$ and $N_j$ large enough.
It is clear that $||F_M-\mathrm{id}||_{C^1}$ converges to 0 uniformly on $\overline D$,
and $F_M(p)$ is a strongly convex boundary point of $F_M(D)$.

$\hfill\square$

\emph{Proof of Theorem \ref{thm:local expose family}:}
For $\zeta\in\partial D$,
denote by $n_\zeta$ the unit outward-pointing normal vector
of $\partial D$ at $\zeta$.
Let $L_\zeta=\{xn_\zeta+iyn_\zeta; x, y\in\mr\}$
and let $\pi_\zeta:\mc^n\rightarrow L_\zeta$ be the orthogonal projection.
For $z\in \mc^n$,
let $t_{\zeta}(z)=x+iy$ if $\pi_\zeta(z-\zeta)=xn_\zeta+iyn_\zeta$
and let $z'_\zeta=(z-\zeta)-\pi_{\zeta}(z-\zeta)$.
Note that $z\mapsto (z'_\zeta, t_\zeta(z))$ is a linear isomorphism
form $\mc^n$ to $(T^\mc_\zeta\partial D)\times\mc$
which maps $\zeta$ to the origin,
where $T^\mc_\zeta\partial D=T_\zeta\partial D\cap iT_\zeta\partial D$
is the complex tangent space of $\partial D$ at $\zeta$.
Let $\rho$ be a defining function of $D$
which is strictly plurisubharmonic on some neighborhood of $\overline D$.
After normalization, near each $\zeta\in\partial D$,
$\rho$ can be expanded as follows:
\begin{equation}\label{eqn:expan of def ftn}
\rho(z)=Re (t_\zeta(z))+Q_\zeta(z-\zeta))+\mathcal L_\zeta(z-\zeta)+o(||z-\zeta||^2),
\end{equation}
where $L_\zeta$ is the Levi form of $\rho$ at $\zeta$ and $Q$ is the complex Hessian
of $\rho$ at $\zeta$.

By some basic results from calculus,
the right hand side of \eqref{eqn:expan of def ftn} can be viewed
as a smooth function defined on some neighborhood of $\partial D\times\overline D$ in $\partial D\times\mc^n$.
Let $p:\partial D\times\overline D\rightarrow \mc$ be the smooth map furnished by Lemma \ref{lem:estimate peak ftn family}.
Then, by the same estimate as in the proof of Theorem \ref{thm:local expose},
we can show that there are positive integer $M$ and $N_j, j=1, \cdots, M$
such that the map $F(\zeta, z):\partial D\times \overline D\rightarrow\mc^n$ given by
\begin{equation*}
\begin{cases}
F(\zeta, z)'_\zeta =z'_\zeta \\
t_\zeta(F(\zeta,z))=t_\zeta(z)+\sum_{j=1}^M\frac{1}{M}Q_\zeta(z-\zeta)p_\zeta^{N_j}(z)
\end{cases}
\end{equation*}
satisfies the required conditions.
$\hfill{\square}$

We now construct a strongly pseudoconvex domain $D\subset\mc^2$
such that the map $F$ in Theorem \ref{thm:local expose}
can not be taken in $\mathrm{Aut_{hol}}(\mc^n)$ for some $p\in\partial D$.

\begin{exm}
Let $D:=\{(z,w)\in\mc^2: |z|^2+|\frac{1}{z}|^2+|w|^2<3\}$.
Then $D$ is a bounded strongly pseudoconvex domain with smooth boundary.
The intersection $A$ of $D$ and the $z$-axis is an annulus.
Let $p$ be an inner boundary point of $A$.
It is clear that there exists a compact set $V\subset D$ such that
the polynomial hull of $V$ contains an open neighborhood of $p$.
So if $F_j\in \mathrm{Aut_{hol}}(\mc^2)$ is a sequence that converges to $\mathrm{id}$ uniformly on $\overline D$,
then $F_j$ converges to $\mathrm{id}$ uniformly on a neighborhood of $p$.
So $F_j(p)$ can not be a strictly convex boundary point of $F_j(D)$ for $j$ large enough.
\end{exm}

\section{The ball model case in $\mc^n$}\label{sec:the ball model}

The main technical problem in proving Theorem \ref{thm:exposing manifolds}
which differs from the results in \cite{Diedrich-Fornaess-Wold13} is to prove
an exposing result for balls, which can later be used to pass from
local to global exposing by approximately gluing.
For $r\in\mr$,
we denote the point $(0, \cdots, 0, r)$ in $\mc^n$ by $p_r$.
For $r>0$ and $a\in\mc^n$,
$B_r(a)$ denotes the ball in $\mc^n$ centered at $a$ with radius $r$,
and we denote by $\mathbb B^n$ the unit ball in $\mc^n$ as usual.
For $r, s\in\mr, r<s$, we let $l_{r,s}$ denote the closed line segment between $p_r$ and $p_s$. \
The main aim of this section is to prove
the following theorem which is a key step
in the proof of Theorem \ref{thm:exposing manifolds}.

\begin{thm}\label{thm:exposing ball case}
Let $r, s$ be positive numbers with $s>r+1$.
Then for any open neighborhood $V$ of $l_{1, s-r}$,
any open neighborhood $U$ of $l_{1, s-r}\cup \overline{B_r(p_s)}$,
and any sufficiently small $\epsilon>0$,
there exits a sequence of maps $\varphi_{\nu,t}:\overline{\mathbb B^n}\rightarrow \mc^n$, $(t\in[0,1])$,
of injective holomorphic maps such that the following hold:
\begin{itemize}
\item[(i)] $\phi_{\nu,t}$ are smooth in $t$ and $\phi_{\nu,0}=Id$,
\item[(ii)] $\phi_{\nu,t}\rightarrow Id$ uniformly on $\overline{\mathbb B^n}\backslash B_\epsilon(p_1)$ for $t\in [0,1]$, as $\nu\rightarrow \infty$,
\item[(iii)] $\phi_\nu(p_1)=p_{s+r}$ for all $\nu\in\mathbb N$, where $\phi_\nu:=\phi_{\nu, 1}$,
\item[(iv)] $\phi_\nu(B_\epsilon(p_1)\cap \overline{\mathbb B^n})\subset V\cup B_r(p_s)\cup\{p_{s+r}\}$,
\item[(v)] $\phi_{\nu,t}(\overline{\mathbb B^n}\cap B_\epsilon(p_1))\subset U$ for all $\nu$ and $t$,
\item[(vi)] the images $\phi_{\nu,t}(\overline{\mathbb B^n})$ are polynomially convex for all $\nu$ and $t$.
\end{itemize}
\end{thm}

\begin{lem}\label{lem:Riemann map convergence}
Let $a<b$ be points on the real line with $b-a>2$,
and let $l$ be the closed line segment between $a+1$ and $b-1$.
Let $\delta_\nu\rightarrow 0$ and let
$\Omega_\nu$ be the $\delta_\nu$ neighborhood of $\overline {D_1(a)}\cup l\cup \overline {D_1(b)}$,
where $D_s(a)$ is the disk in $\mc$ centered at $a$ with radius $s$.
There exist injective holomorphic maps $f_\nu, g_\nu: \overline{\mathbb D}\rightarrow \mc$
such that the following holds:
\begin{itemize}
\item[(i)] $f_\nu(\overline{\mathbb D}\cap \mr)\subset\mr, g_\nu(\overline{\mathbb D}\cap \mr)\subset\mr$,
           and $f_\nu(\mathbb D)=g_\nu(\mathbb D)\subset \Omega_\nu$ for all $\nu$,
\item[(ii)] $f_\nu(-1)=g_\nu(-1)=a-1, f_\nu(1)=g_\nu(1)=b+1$ and $f_\nu(0)=a, g_\nu(0)=b$,  for all $\nu$,
\item[(iii)] $f_\nu(z)\rightarrow z+a$ uniformly on some fixed neighborhood of $\overline{\mathbb D}\backslash D_\epsilon(1)$ for
any small $\epsilon>0$;
\item[(iv)] $g_\nu(z)\rightarrow z+b$ uniformly on some fixed neighborhood of $\overline{\mathbb D}\backslash D_\epsilon(-1)$ for
any small $\epsilon>0$;
\item[(v)] $f_\nu(z)=g_\nu(m_\nu(z))$ for all $\nu$, where
     $$m_\nu(z)=\frac{z-r_\nu}{1-r_\nu z}$$
     with $r_\nu\rightarrow 1$ as $\nu\rightarrow \infty$.
\end{itemize}
\end{lem}
\begin{proof}
For simplicity we assume $a=-2,b=2$.
We will construct maps $f_\nu$ that satisfy (i)-(iii)
and $f_\nu(\overline{\mathbb D})$ is invariant under the transformation $z\mapsto -z$.
By symmetry, if we put $g_\nu(z)=-f_\nu(-z)$,
then all conditions about $f_\nu$ and $g_\nu$ in the lemma hold.
The convergence of $r_\nu$ to 1 in (v) will be guaranteed by
the convergence of $f_\nu$ as presented in (iv).

We now construct such $f_\nu$.
By Mergelyan's Theorem there exists a sequence of embeddings
$\phi_\nu:\overline{D_1(-2)}\cup l\cup\overline{D_1(2)}\rightarrow\mathbb C$
such that $\phi_\nu\rightarrow\mathrm{id}$ on $\overline{D_1(-2)}$, $\phi_\nu(l)$
shrinks to the point $-1$, as $\nu\ra\infty$,
and such that $\phi_\nu$ is are as close as we like to a translation and
a scaling on $\overline{D_1(2)}$, such that
the diameter of $\phi_\nu(D_1(2))$ shrinks to zero.
We may also interpolate to get $\phi_\nu(-2)=-2,\phi_\nu'(-2)=1$).
Replacing $\phi_\nu$ by the map given by $z\mapsto \frac{\phi_\nu(z)+\overline{\phi_\nu(\bar {z})}}{2}$
if necessarily, we can assume that $\phi_\nu(z)=\overline{\phi_\nu(\bar {z})}$,
and hence $\phi_\nu$ maps real numbers to real numbers
Let $W_\nu=D_1(-2)\cup l(\delta'_\nu)\cup D_1(2)$,
where $l(\delta'_\nu)$ is the $\delta'_\nu$-neighborhood of $l$ and $0<\delta'_\nu<\delta_\nu$.
If the $\delta'_\nu$ are chosen small enough, then $\phi_\nu(W_\nu)$ converges
to the ball $D_1(-2)$ in the sense of Goluzin.

Let $\psi_\nu:\mathbb D\rightarrow\phi_\nu(W_\nu)$ be the Riemann map
with $\psi_\nu(0)=-2, \psi_\nu'(0)>0$,
we have that $\psi_\nu(z)\rightarrow z-2$ uniformly on $\overline{\mathbb D}$
(see \cite{Goluzin69}, Theorem 2, p. 59.).  Morover, $\psi_\nu(z)=\overline{\psi_\nu(\overline z)}$.
Then it is clear that $\tilde f_\nu:=\phi_\nu^{-1}\circ\psi_\nu$ satisfy (i), (ii), and (iii).

\end{proof}

\begin{lem}\label{lem:estimate mobus map}
Let $0<r<1$ and define $\varphi_r(z)=\frac{z-r}{1-rz}$.
Then either $|\varphi_r(z)|<|z|$ or $Re(\varphi_r(z))<0$
for all $z\in\mathbb D$.
\end{lem}
\begin{proof}
Note first that any circle $\Gamma_a=\{|z|=a\}$
with $a>0$ is mapped by $\varphi_r$ to a circle
symmetric with respect to the real line,
and $\varphi_r(a)<a$.
A straightforward computation shows that
$|\varphi_r(a)-\varphi_r(-a)|=2a\frac{1-r^2}{1-a^2r^2}$
and so the radius of $\varphi_r(\Gamma_a)$ is less than $a$.
Since also $\varphi_r(a)<a$ this implies that $\varphi_r(\Gamma_a)\cap\Gamma_a$
is either empty or is contained in $\{Re(z)<0\}$.
\end{proof}

\begin{lem}\label{lem:identify diff balls}
Let $r_j,s_j$ be real numbers for $j=1,2$ with $s_j>r_j+1$.  Then there exists a sequence
\begin{equation}
\phi_\nu:\overline{\mathbb B^n}\cup l_{1,s_1}\cup\overline{B_{r_1}(s_1)}\rightarrow\mathbb C^n
\end{equation}
of injective holomorphic maps
such that the following hold:
\begin{itemize}
\item[(i)] $\phi_\nu\rightarrow Id$ uniformly on $\overline{\mathbb B^n}$,
\item[(ii)] $\phi_\nu(z)\rightarrow \psi(z)=p_{s_2}+\frac{r_2}{r_1}(z-p_{s_1})$ uniformly on $\overline {B_{r_1}(p_{s_1})}$,
\item[(iii)]  $\phi_\nu$ converges uniformly on $l_{1,s_1}$ to a smooth embedding of $l_{1,s_1}$ onto $l_{1,s_2}$.
\end{itemize}
Moreover, the maps $\phi_\nu$ are of the form $\phi_\nu(z)=(h_\nu(z_n)\cdot z',f_\nu(z_n))$, where $h_\nu,f_\nu$
are holomorphic functions in one variable.
Finally, the maps $\phi_\nu$ can be chosen to match $\psi$ to any given order at the point $p_{s_1+r_1}$.
\end{lem}
\begin{proof}
By Mergelyan's theorem there exists one variable functions $f_\nu(z_n)$ satisfying
(i)-(iii) on the intersection with the $z_n$-axis.  Again by Mergelyan's theorem there
exist a sequence $h_\nu(z_n)$ converging to the constant function $r_2/r_1$ on
the disk of radius $r_1$ centred at $s_1$, to $1$ on the closed unit disk, and
some fixed real function increasing/decreasing from $r_2/r_1$ to $1$ on the line segment
connecting the two disks.
\end{proof}

We now give the proof of Theorem \ref{thm:exposing ball case}.
\begin{proof}
We prove it first without the parameter $t$, \emph{i.e.}, we prove
the existence of the sequence $\phi_{\nu,1}$.  By Lemma \ref{lem:identify diff balls},
it is enough to prove this for some special case of $r, s$.
We take the special case that $r=1.5, s=3$.
Let $f_{\nu}$ be maps furnished by Lemma \ref{lem:Riemann map convergence} with $a=0, b=3.5$.
We define $\phi_{\nu}(z):=(z_1, \cdots, z_{n-1}, f_{\nu}(z_n))$.
Then the properties (ii), (iii), (v) and (vi) hold (we have that (vi) holds because $f_\nu^{-1}$ is approximable by entire maps).  \

We now prove (iv).  Since $\epsilon$ can be chosen arbitrarily small, the only place where the inclusion in (iv)
can fail is near the point $p_{4.5}$.
Consider first the sequence
$$
\tilde\phi_\nu(z)=(z_1, \cdots, z_{n-1}, g_{\nu}(z_n)).
$$
Then $\tilde\phi_\nu(z)\rightarrow z + p_b$ uniformly on $\overline{\mathbb B^n}\setminus B_\epsilon(p_{-1})$,
and since $g_\nu'(1)$ is real (and goes to 1) we have that $\tilde\phi({\mathbb B^n}\setminus B_\epsilon(p_{-1}))$
is eventually contained in $B_{1.5}(p_3)$.  Now let $z=(z_1,...,z_n)$ be a point in $\overline{\mathbb B^n}\cap B_\epsilon(p_1)$.
If $\mathrm{Re}(m_\nu(z_n))\leq 0$ then $\phi_\nu(z)$ is far away from the point $p_{4.5}$ since $\phi_\nu(z)=\tilde\phi_\nu(z_1,...,z_{n-1},g_\nu(m_\nu(z_n)))$.  If $\mathrm{Re}(m_\nu(z_n))>0$ it follows from Lemma \ref{lem:estimate mobus map}
that $(z_1,...z_{n-1},m_\nu(z_n))\in\mathbb B\setminus B_\epsilon(p_{-1})$, and so $\phi_\nu(z)\in B_{1.3}(3)$.  \

To construct isotopies we simply define
$$
\phi_{\nu,t}(z)=\frac{1}{t}\phi_\nu(tz).
$$
Note first that we can choose $\epsilon>0$ arbitrarily small, and that by the construction, $\mathbb B^n\cap B_\epsilon(p_1)$
gets mapped by $\phi_\nu$ into a relatively compact subset $\tilde U$ of $U$ which is independent of $\epsilon$.
Write
$$
\phi_{\nu}(z)=A_\nu(z) + G_\nu(z),
$$
where $A_\nu\rightarrow\mathrm{Id}$ and $\|G_\nu(z)\|\leq\delta_\nu\|z\|^2$ for $\|z\|<1-\epsilon$ where $\delta_\nu\rightarrow 0$.
Consider a point $z\in B_\epsilon(p_1)\cap\mathbb B^n$.  If $t<1-\epsilon$ then $\phi_{\nu,t}(z)=A_\nu z + \frac{1}{t}G_\nu(tz)$,
with $\|G_\nu(tz)\|\leq\delta_\nu t^2\|z\|^2$, so if $\nu$ is large we are in $U$.  If $t\geq 1-\epsilon$ we consider
two cases.  If $tz\notin B_\epsilon(p_1)$ we may assume that $\phi_\nu$ is as close to the identity as we like,
and so we are still in $U$.   If $tz\in B_\epsilon(p_1)$ then $\phi_\nu$ maps $tz$ into $\tilde U$, and $\frac{1}{t}\tilde U\subset U$
provided $\epsilon$ was chosen small enough.
\end{proof}

\section{Exposing a single boundary point}\label{sec:expose single point}

The aim of this section is to prove Theorem \ref{thm:exposing manifolds}.
Recall that a pair $(A,B)$ of compact sets in a complex space $X$
is called a \emph{Cartan pair} if
$A, B, A\cap B, A\cup B$ all admit Stein neighborhood bases in $X$
and $\overline{A\backslash B}\cap \overline{B\backslash A}=\emptyset$.
The following lemma due to Forstneri\v{c} will be used to (approximately)
glue locally defined exposing maps to define global ones.

%
%
\begin{lem} [\cite{Forstneric03}\cite{Forstneric13}]\label{lem:gluing tech}
Assume that $X$ is a complex space and $X'$ is a closed complex subvariety
of $X$ containing the singular locus $X_{\mathrm{sing}}$.
Let $(A, B)$ be a Cartan pair in $X$ such that
$C:=A\cap B\subset X\backslash X'$.
For any open set $\tilde C\subset X$ containing $C$
there exist open sets $A'\supset A, B'\supset B, C'\supset C$ in $X$,
with $C'\subset A'\cap B'\subset\tilde C$,
satisfying the following property.
For every number $\eta>0$ there exists a number $\epsilon_\eta>0$
such that for each holomorphic map $\gamma:\tilde C\rightarrow X$
with $\text{dist}_{\tilde C}(\gamma, \text{Id})<\epsilon_\eta$
there exist biholomorphic maps
$\alpha=\alpha_\gamma: A'\rightarrow \alpha(A')\subset X$ and
$\beta=\beta_\gamma:B'\rightarrow\beta(B')\subset X$
satisfying the following properties:
\begin{itemize}
\item[(i)] $\gamma\circ\alpha=\beta$ on $C'$,
\item[(ii)]$\text{dist}_{A'}(\alpha, \text{Id})<\eta$ and $\text{dist}_{B'}(\beta, \text{Id})<\eta$, and
\item[(iii)] $\alpha$ and $\beta$ are tangent to the identity map to any given finite order
             along the subvariety $X'$ intersected with their respective domain.
\end{itemize}
Moreover, the maps $\alpha_\gamma$ and $\beta_\gamma$ can be chosen to depend continuously on $\gamma$
such that $\alpha_{\text{Id}}=\text{Id}$ and $\beta_{\text{Id}}=\text{Id}$.
\end{lem}

We start by embedding a neighbourhood of $\gamma$ suitably into $\mathbb C^n$, where we will
define local exposing maps.

\begin{lem}\label{coordinategamma}
There exists an open Stein neighbourhood $W$ of $\gamma$, an
open neighbourhood $U_\zeta$ of $\zeta,  $and a holomorphic
embedding $\phi:W\rightarrow\mathbb C^n$ ($n=\mathrm{dim}(X)$) such that the following holds:
\begin{itemize}
\item[(i)] $\phi(\zeta)=0$,
\item[(ii)] $\phi((U_\zeta\cap K)\setminus\{\zeta\})\subset\{z\in\mathbb C^n:2\mathrm{Re}(z_n)+\|z\|^2<0\}$
\end{itemize}
\end{lem}
\begin{proof}
First let $M\subset X$ be a totally real manifold (with boundary) of dimension $n$ that contains the curve $\gamma$,
and choose a smooth embedding $g:M\rightarrow \mathbb R^n\subset\mathbb C^n$.  Then $M$ admits a Stein neighbourhood
$W_0$ such that $g$ may be approximated in $C^1$-norm on $M$ by holomorphic maps $\phi_0:W_0\rightarrow\mathbb C^n$.
So $\phi_0$ may be taken to be an embedding of an open Stein neighbourhood $W$ of $\gamma$ into $\mathbb C^n$.
By assumption there is a (local) strictly pseudoconvex hypersurface $\Sigma\subset\mathbb C^n$ and
an open set $U_\zeta$ such that $\Sigma$ touches $\phi_0(U_\zeta\cap K)$ only at the point $\phi_0(\zeta)$.
Let $F\in\mathrm{Aut_{hol}}\mathbb C^n$ such that $F(\Sigma)$ is strictly convex near $F(\phi_0(\zeta))$.
After a translation and scaling, the map $\phi=F\circ\phi_0$ will ensure the conclusions of the lemma.
\end{proof}

We may now modify $\gamma$ such that near the origin, the curve $\tilde\gamma:=\phi(\gamma)$ coincides
with the line segment $l_{0,r}$ for some $r>0$, and such that $\tilde\gamma$ is perpendicular to the (local)
hypersurface $\tilde H:=\phi(W\cap H)$ where they intersect.  We keep the notation $\gamma$ for
the modified curve.  Let $\gamma_\zeta$ denote the
piece $\phi^{-1}(l_{0,r})$ of $\gamma$.

\begin{lem}
There exists a
a compact set $K'\subset X$ with $K\subset\mathrm{int}(K')$
such that for any $\epsilon>0$ there exists an open neighbourhood $\Omega$
of $K'\cup\gamma_\zeta$
and a holomorphic embedding $\psi:\Omega\rightarrow X$
such that the following holds:
\begin{itemize}
\item[(i)] $\mathrm{dist}(z,\psi(z))<\epsilon$ for all $z\in K'$,
\item[(ii)] $\psi(\gamma_\zeta)\subset V\cap W$,
\item[(iii)] $\psi(\gamma_\zeta)\cap H=\psi(\phi^{-1}(p_r))$ and the intersection is perpendicular
to $H$ in the coordinates furnished by $\phi$.
\end{itemize}
\end{lem}
\begin{proof}
We will first define a map $\tilde\psi$ accomplishing (i)-(iii) in the local coordinates furnished by $\phi$,
and then we will (approximately) glue the map $\phi^{-1}\circ\tilde\psi\circ\phi$ to the identity
map on $K'$.    \

We start by defining a suitable Cartan pair.  Writing  $\phi=(\phi_1,...,\phi_n)$ note that
the pluriharmonic function  $\rho:=\mathrm{Re}(\phi_n)$ extends to a plurisubharmonic function on
an open neighbourhood of $K$ which agrees with $\rho$ on an open neighbourhood of $\zeta$,
and which is negative and uniformly bounded away from zero away from where they agree.
We keep the notation $\rho$ for the extended function.
Then for a sufficiently small smoothly bounded strictly pseudoconvex open neighbourhood $\Omega'$ of $K\cup\gamma_\zeta$
and sufficiently small $\tau>0$, the sets
$$
A_\tau=\overline\Omega'\cap\{\rho\geq-2\tau\} \mbox{ and }  B_\tau=\overline\Omega'\cap\{\rho\leq -\tau\}
$$
define a Cartan pair $(A_\tau,B_\tau)$.  Let $\tilde A_\tau=\phi(A_\tau)$ and $\tilde B_\tau=\phi(B_\tau\cap W)$.  If $\Omega'$
and $\tau$ are chosen sufficiently small we have that $A_\tau\cap B_\tau\subset W$ and  $\tilde A_\tau\cap \tilde B_\tau\subset B_{r/2}(0)$.  Set $\tilde C_\tau=B_{r/2}(0)$.
Fix $K'$ a Stein compact such that $K'\subset B_\tau\cup\phi^{-1}(B_{r/2})$.  \

Now let $D$ be an open neighbourhood of $\overline B_{r/2}(0)\cup l_{0,r}$ with $\tilde g:D\rightarrow\mathbb C^n$
a smooth embedding such that $\tilde g=\mathrm{id}$ near $\overline B_{r/2}(0)$, $\tilde g$ stretches
$l_{0,r}$ to cover $\tilde\gamma$, $\tilde g^{-1}(\tilde H)$ is perpendicular to $l_{0,r}$, and such that
$\tilde g$ is $\overline\partial$-flat to order one along $l_{0,r}$.   Then $\tilde g$ is
uniformly approximable on $\overline B_{r/2}(0)\cup l_{0,r}$ in $\mathcal C^1$-norm with jet interpolation
at the point $p_r$.   So there exists a holomorphic embedding $\tilde h:\overline B_{0,r/2}(0)\cup l_{0,r}\rightarrow\mathbb C^n$,
as close to the identity as we like on $\overline B_{r/2}(0)$, with the image $\tilde h(l_{0,r})$ as close
as we like to $\tilde\gamma$ and with $\tilde h^{-1}(\tilde H)$ perpendicular to $l_{0,r}$.  Set $h:=\phi^{-1}\circ\tilde h\circ\phi$.
Now if $\tilde h$ is close enough to the identity on $B_{r/2}$ we have that Lemma \ref{lem:gluing tech}
furnishes maps $\alpha$ and $\beta$ such that the map $\psi$ defined as $\psi:=h\circ\alpha$ on $A_\tau$
and $\psi=\beta$ on $B_\tau$ will satisfy the claims of the lemma.

\end{proof}

\emph{Proof of Theorem \ref{thm:exposing manifolds}:}  We use the extended function $\rho$ from
the proof of the previous lemma to define a Cartan pair:
$$
A_\tau=K\cap\{\rho\geq-2\tau\} \mbox{ and }  B_\tau=K\cap\{\rho\leq -\tau\}
$$
Then for $\tau$ small we have that $\phi$ maps $A_\tau$ into $B_1(p_{-1})\cup\{0\}$, and
$C_\tau=A_\tau\cap B_\tau$ gets mapped into the ball $B_1(p_{-1})$.  Choose a small
$\delta>0$ such that the ball $B_\delta(p_{r-\delta})$ touches $\tilde\psi^{-1}(H)$ only
at the point $p_r$.  Now Theorem \ref{thm:exposing ball case} applied
to the dumbbell $B_1(p_{-1})\cup l_{0,r-\delta}\cup B_\delta(p_r)$ furnishes local
exposing maps $\tilde\phi_\nu:\overline B_1(p_{-1})\rightarrow\mathbb C^n$ such
that the maps $\phi_\nu:=\phi^{-1}\circ\tilde\phi_\nu\circ\phi$ may be approximately glued to
the identity map over a neighbourhood of $C_\tau$ to conclude the proof of the first part of Theorem \ref{thm:exposing manifolds}
by setting $f=\psi\circ\phi_\nu$.  \

For the last part we now assume that $K\subset\mathbb C^n$ is polynomially convex.   Let $\rho$
be a non-singular strictly plurisubharmonic function on an open neighbourhood $U_\zeta$ of $\zeta$
such that $\rho(\zeta)=0$ and $\rho<0$ on $(K\cap U_\zeta)\setminus\{\zeta\}$.  Then
for any sufficiently small Runge and Stein neighbourhood $\Omega$ of $K$ the function
$\rho$ extends to a plurisubharmonic function on $\Omega$ which is strictly negative on $\Omega\setminus U_\zeta$,
and so $\Omega'=\Omega\cap\{\rho<0\}$ is a Runge and Stein domain with $\zeta\in b\Omega$.   Moreover, choosing
$\overline\Omega$ to have a Runge and Stein neighbourhood basis, we have that $\overline\Omega'$ will
have a Runge and Stein neighbourhood basis, and so by Theorem 1.3 in \cite{Diedrich-Fornaess-Wold13}
there exists $F\in\mathrm{Aut_{hol}}\mathbb C^n$ such that $F(K\setminus\{\zeta\})\subset\mathbb B^n$
and $F(\zeta)=p_1$.   Since we can conjugate by $F$ we may as well assume that $K\setminus\{\zeta\}\subset\mathbb B^n$
and that $\zeta=p_1$.  \

My slightly modifying $\gamma$ we may assume that $\gamma$ agrees with $l_{1,1+r}$ near $p_1$
and that it is perpendicular to $H$ where they intersect.  First let $D$ be an open neighbourhood
of $\overline{\mathbb B^n}\cup l_{1,1+r}$ and $g:[0,1]\times D\rightarrow\mathbb C^n$
be an isotopy of embeddings such that $g_0=\mathrm{id}$, $g$ stretches $l_{1,1+r}$
along $\gamma$, $g_t=\mathrm{id}$ near $\overline{\mathbb B^n}$ for all $t$, $g^{-1}(H)$
is perpendicular to $l_{1,1+r}$, and $g_t$ is $\overline\partial$-flat on $l_{1,1+r}$ to order one.  Then
the whole parameter family $g_t$ may be approximated by a parameter family $h:[0,1]\times\mathbb C^n$
of holomorphic maps in $C^1$-norm on $\overline{\mathbb B^n}\cup l_{1,1+r}$ with jet interpolation at
the end point of $l_{1,1+r}$.   So $h_t$ is a family of holomorphic embeddings on some open neighbourhood
of $\overline{\mathbb B^n}\cup l_{1,1+r}$, and by the Anders\'{e}n-Lempert theory $h_1$ may
be approximated by a holomorphic automorphism $G$ with interpolation at the end point of $l_{1,1+r}$. \

Let $\tilde H=G^{-1}(H)$ which is now perpendicular to $l_{1,1+r}$.   Choose $s>0$ such that
$B_{s}(p_{1+r-s})$ touches $\tilde H$ only at a single point.   Now Theorem \ref{thm:exposing ball case}
applied to the dumbbell $\mathbb B^n\cup l_{1,1+r}\cup B_{s}(p_{1+r-s})$ furnishes
a one-parameter family of maps $\phi_t$ such that $G\circ\phi_1$ is
the exposing map we are after, except that $\phi_1$ is not an automorphism.  However, by (vi),
$\phi_t(\overline{\mathbb B^n})$ is polynomially convex for each $t$, and so $\phi_1$ is approximable
by holomorphic automorphisms.  The proof is complete.

$\hfill\square$

\section{Regularity of exposing maps with parameters}\label{sec:parametric exposing}

In this section we will prove Theorem \ref{thm:exposing family manifolds}.  Before
we proceed with the proof, we give a further discussion on parametric exposing.

\begin{defn}\label{def:exposing things}
Let $D$ be a bounded domain in $\mc^n$ with smooth boundary.
Then we say that $D$ satisfies:\\
1) \emph{Property (E)}, if there exists a continuous map
$F:\partial D\times\bar D\rightarrow \mc^n$
such that for each fixed $p\in\partial D$
the map $F_\zeta:=F(\zeta,\cdot):\bar D\rightarrow \mc^n$
is an exposing of $D$ at $\zeta$; or\\
2) \emph{Property (AE)}, if there exists a continuous map
$F:\partial D\times\mc^n\rightarrow \mc^n$
such that for each fixed $\zeta\in\partial D$
we have that $F_\zeta\in\mathrm{Aut_{hol}}\mathbb C^n$
is an exposing of $D$ at $\zeta$; or\\
3) \emph{Property (ASE)}, if there exists a continuous map
$F:\partial D\times\mc^n\rightarrow \mc^n$
such that for each fixed $\zeta\in\partial D$
 we have that $F_\zeta\in\mathrm{Aut_{diff}}\mathbb C^n$
 and $F_\zeta(D)$ is exposed at $F_\zeta(\zeta)$.
\end{defn}

In the case 2) we say that $F_\zeta$ is ambient exposing, and in case 3)
we say that $F_\zeta$ is ambient smoothly exposing.
Note that Property (ASE) has nothing
to do with the complex structure of $D$ and $\mc^n$.
In \cite{Deng-Guan-Zhang13},
the following question was proposed:

\begin{que}
Let $D$ be a bounded strongly
pseudoconvex domain in $\mc^n$ with smooth boundary.
Does $D$ satisfies Property (E)?
\end{que}

The following theorem, which can be viewed as a parametric version of Theorem \ref{thm:exposing manifolds},
partially answers the above question.

\begin{thm}\label{thm:exposing family}
Let $D\subset\mc^n (n>1)$ be a bounded strongly
pseudoconvex domain with smooth boundary.
If $D$ satisfies Property (ASE),
then for any positive numbers $\epsilon, \delta$,
there exists a continuous map $f:\partial D\times\overline D\rightarrow \mc^n$
such that for each $\zeta\in\partial D$:
\begin{itemize}
\item[1)] $f_\zeta(\cdot)=f(\zeta,\cdot):\overline D\rightarrow \mc^n$
   is holomorphic and injective;
\item[2)] $f_\zeta(D)$ is exposed at $f_\zeta(\zeta)$; and
\item[3)] $||f_\zeta(z)-z||<\epsilon$ for $z\in\overline D\backslash B_\delta(\zeta)$,
   where $B_\delta(\zeta)=\{z\in\mc^n; ||z-\zeta||<\delta\}$.
\end{itemize}
In particular,
$D$ satisfies Property (E).
If in addition $\overline D$ is polynomially convex,
then $f$ can be taken to be a smooth map $f:\partial D\times \mc^n\rightarrow \mc^n$
such that $f_\zeta\in Aut(\mc^n)$ for all $\zeta\in\partial D$.
In particular,
$D$ satisfies Property (AE).
\end{thm}

When trying to apply Theorem \ref{thm:exposing family manifolds} to Question 1,
we meet a topological obstruction.
Even we originally just interested in exposing but not ambient exposing,
it turns out that Property (ASE) is needed in our argument.
It is clear that Property (ASE) is a necessary condition for Property (AE),
but we don't know if it is also necessary for Property (E).

For the special case when $\overline D$ is diffeomorphic to the closed unit ball,
we can prove the following

\begin{thm}\label{thm:diff to ball}
Let $D$ be a bounded strongly pseudoconvex domain
in $\mc^n$ with smooth boundary.
Assume that either $D$ is diffeomorphic to the unit ball $B^n$ and $n\geq 3$,
or the closure $\overline D$ of $D$ is diffeomorphic to the closed unit ball and $n=2$.
Then $D$ satisfies Property (E).
If in addition $\bar D$ is polynomially convex,
then $D$ satisfies Property (AE).
\end{thm}

\begin{rem}
It is known that any smoothly embedded $S^{n-1}$ in $\mr^n$
bounds a domain that is diffeomorphic to the unit ball except $n\neq 4$ (see \cite{Milnor65}).
The case for $n=4$ is still open.
So for $n\neq 2$, the condition in Theorem \ref{thm:diff to ball}
that $D$ is diffeomorphic to $B^n$ can be replaced by that
$\partial D$ is diffeomorphic to $S^{2n-1}$.
\end{rem}

\subsection{A parametric version of Forstneri\v{c}'s  splitting lemma}
We first introduce some notations.
Let $D$ be a bounded strictly pseudoconvex domain in $\mc^n$
with $C^2$-smooth boundary and $\delta$ be a sufficiently small positive number.
For $\zeta\in\partial D$ we define:
\begin{equation*}
\begin{split}
& A_\zeta=\{z\in\overline D: ||z-\zeta||\leq\delta\},\\
& B_\zeta=\{z\in\overline D: ||z-\zeta||\geq\delta/2\},\\
& C_\zeta=A_\zeta\cap B_\zeta.
\end{split}
\end{equation*}
For a closed subset $X$ in $\mc^n$ and $a>0$,
we set $X(a):=\{z\in\mc^n:||z-w||\leq a\ \text{for\ some}\ w\in X\}$
the $a$-neighborhood of $X$.

The following result is the main result in the thesis of Lars Simon (\cite{Siom14})
\begin{lem}\label{lem:para splitting}
Let $D$ and $\delta$ as above.
If $\tau>0$ is small enough (this depends on $\delta$),
then for any $\eta>0$ there exists $\epsilon>0$ such that the following holds.
If $\mu>5\tau$ and if $\gamma_\zeta$ is a family of injective
holomorphic maps $\gamma_\zeta:C_\zeta(\mu)\ra \mc^n$ satisfying
\begin{itemize}
\item $\gamma:\{(z,\zeta):z\in C_\zeta(\mu), \zeta\in\partial D\}\ra\mc^n,\ (z,\zeta)\mapsto \gamma_\zeta(z)$ is continuous,
\item $dist_{C_\zeta(\mu)}(\gamma_\zeta, Id)<\epsilon$ for all $\zeta\in\partial D$,
\end{itemize}
then there exist families $\{\alpha_\zeta\}_{\zeta\in\partial D}$ and $\{\beta_\zeta\}_{\zeta\in\partial D}$
of injective holomorphic maps $\alpha_\zeta:A_\zeta(2\tau)\ra\mc^n$
and $\beta_\zeta:B_\zeta(2\tau)\ra\mc^n$ with the following properties:
\begin{itemize}
\item[(1)] for all $\zeta\in\partial D$ we have $\gamma_\zeta=\beta_\zeta\circ\alpha^{-1}_\zeta$ on $C_\zeta(\tau)$,
\item[(2)] $dist_{A_\zeta(2\tau)}(\alpha_\zeta, Id)<\eta$ and $dist_{B_\zeta(2\tau)}(\beta_\zeta, Id)<\eta$,
\item[(3)] the maps $\alpha$ and $\beta$ are continuous, where
\begin{equation*}
\begin{split}
&\alpha: \{(z,\zeta)\in\mc^n\times\partial D:z\in A_\zeta(2\tau)\}\ra\mc^n,\ (z,\zeta)\mapsto \alpha_\zeta(z),\\
&\beta:  \{(z,\zeta)\in\mc^n\times\partial D:z\in B_\zeta(2\tau)\}\ra\mc^n,\ (z,\zeta)\mapsto \beta_\zeta(z).
\end{split}
\end{equation*}
\end{itemize}
Moreover, the maps $\alpha_\zeta$ may be constructed to match the identity to any given order at $\zeta$,
and if we are given a parameter family $\gamma_{\zeta,t}$ with $t\in [0,1]$ we may
obtain $\alpha_{\zeta,t}$ and $\beta_{\zeta,t}$ also jointly continuous in $(\zeta,t)$.
\end{lem}
We note that the two last points were not a part of the statement in \cite{Siom14}, but
they follow from the proof.

\subsection{Exposing along normal directions}

Let $\Omega\subset\mathbb C^n$ be a domain with a strictly plurisubharmonic defining function $\rho$,
and fix $p\in \partial\Omega$.  By choosing a smooth family orthogonal frames in $T_\zeta \partial\Omega$ for
$\zeta$ near $p$ we may construct a smooth family of locally injective holomorphic maps $g_\zeta(z)$ such that
$g_\zeta(\Omega)$ has a defining
function
\begin{equation}
\rho_\zeta(z)= 2\mathrm{Re}(z_1 + Q_\zeta(z)) + \mathcal L_\zeta(z) + \mathrm{h.o.t}
\end{equation}
near the origin.  Define $G_\zeta(z)=(z_1-Q_\zeta(z),z_2,...,z_n)$, so that $H_\zeta=G_\zeta\circ g_\zeta$ maps
$\partial\Omega$ to a strictly convex surface near the origin.   We may cover $\partial\Omega$ by finitely
many open sets $U_j$ such that we have such families of maps $H^j_\zeta$ defined on each $U_j$,
and they all coincide modulo a unitary change of coordinates in $\{z_1=0\}$.  Choose a small
$a>0$, let $l_a$ denote the line segment between $0$ and $a$ in the $z_1$-axis,
and let $\eta_\zeta=(H^j_\zeta)^{-1}(l_a)$.  For small $b,c>0$ we let $B(b,\zeta)$ denote
the ball $(H^j_\zeta)^{-1}(B_b(a))$, and $V_\zeta=(H^j_\zeta)^{-1}(l_a(c))$, the inverse image of the open $c$-tube around
$l_a$. \

\begin{thm}\label{normexpose}
Let $\Omega\subset\mathbb C^n$ be a domain with a strictly plurisubarmonic defining function $\rho$
and with objects defined above.  Then for $\epsilon_1,\delta_1>0$ and $a,b,c>0$ small enough there
exists a continuous map $F:\partial\Omega\times\overline\Omega\subset\mathbb C^n$ such
that the following holds.
\begin{itemize}
\item[(i)] For each $\zeta\in\partial\Omega$ the map $F_\zeta=F(\zeta,\cdot):\overline\Omega\rightarrow\mathbb C^n$
is a holomorphic embedding,
\item[(ii)] $F_\zeta(\zeta)=(H^j_\zeta)^{-1}(a,0,\cdot,0)=:q_\zeta$,
\item[(iii)] $F_\zeta(\overline\Omega\cap B_{\delta_1}(\zeta))\subset V_\zeta\cup B(b,\zeta)\cup\{q_\zeta\}$, and
\item[(iv)] $\|F(\zeta,\cdot)-\mathrm{id}\|_{\overline{\Omega}\setminus B_{\delta_1}(\zeta)}<\epsilon_1$.
\end{itemize}
Moreover, if $\overline\Omega$ is polynomially convex we may achieve that $F_\zeta\in\mathrm{Aut_{hol}}\mathbb C^n$
for each $\zeta$.
\end{thm}
\begin{proof}
We first define local maps achieving (i)-(iii), and then these will be approximately glued
to the identity map using Lemma \ref{lem:para splitting} above.  \

By scaling we may assume that all (local) images $H^{j}_\zeta(\Omega)$
are contained in the ball $B=\{z:2\mathrm{Re}z_1 + \|z\|^2<0\}$ of radius one
centered at the point $(-1,0,...,0)$.  Let $\phi_{\nu,t}$ be the
maps from Theorem \ref{thm:exposing ball case} defined for the ball
$B$, the line segment $l_a$, the ball of radius $b$ centred at the point $a$,
and the set $V=l_a(c)$.   Since $\phi_{\nu,t}$ are one variable maps,
the maps defined by $\tilde F_\zeta(z):=(H^j_\zeta)^{-1}\circ\phi_{\nu}\circ H^j_\zeta$
are well defined independently of $j$, and $\tilde F_\zeta$ is a locally defined map (near each $\zeta$)
achieving (i)-(iii) for $\nu$ large enough.  \

Now choose $\delta>0$ small enough, set $\gamma_\zeta=\tilde F_\zeta|_{C_\zeta(\mu)}$,
and let $\alpha_\zeta,\beta_\zeta$ be the maps from Lemma \ref{lem:para splitting}.
Then the maps $F_\zeta$ defined by $\tilde F_\zeta\circ\alpha_\zeta$ on $A_\zeta(2\tau)$
and $\beta_\zeta$ on $B_\zeta(2\tau)$ are globally defined on $\Omega$ and
will satisfy (i)-(iv) as long as all constants involved are small enough.  \

For the last part, assuming polynomial convexity, note that $F_\zeta$ is
isotopic to the identity map through holomorphic embeddings
by setting $\tilde F_{\zeta,t}:=(H^j_\zeta)^{-1}\circ\phi_{t,\nu}\circ H^j_\zeta$,
including the $t$-parameter from Theorem \ref{thm:exposing ball case}, and
invoking the last statement in Lemma \ref{lem:para splitting}.  Following the
arguments in Section 5 in \cite{Diedrich-Fornaess-Wold13} each  image
$F_{\zeta,t}(\overline\Omega)$ will be polynomially convex, and so by
the Anders\'{e}n-Lempert theorem with parameters and jet-interpolation,
the map $F_{\zeta,1}$ may be approximated by a family of holomorphic
automorphisms.   Note that in the usual statements of the Anders\'{e}n-Lempert
theorems, isotopies are required to by of class $C^2$, but it is quite simple
to work with $C^0$-isotopies instead, see \cite{AndristWold2015} for details.
\end{proof}

\subsection{Proof of Theorem \ref{thm:exposing family manifolds}}

Let $\eta_\zeta$ denote the arcs from the previous section with $a<<1$.
We leave it to the reader to convince himself/herself that the arcs $\gamma_\zeta$
may be modified so that each of them parametrises $\eta_\zeta$ over
an interval $[0,s]$, and so that $\gamma_\zeta$ is perpendicular to $H$.
Let $\psi_t:\partial\Omega\times\overline\Omega\cup(\bigcup_\zeta \{\zeta\}\times\eta_\zeta)\rightarrow\partial\Omega\times\mathbb C^n$ be a fiberpreserving smooth map, $t\in [0,1]$, that is the identity near $\partial\Omega\times\overline\Omega$,
and stretches each $\eta_\zeta$ to cover $\gamma_\zeta$.  According to \cite{KutzschebauchWold}
Section 5 we have that $\psi_t$ may be approximated by an isotopy of injective
holomorphic maps such that $\psi_1(\eta_\zeta)$ is still perpendicular to $H$, keeping
the notation $\psi_t$ for the approximation.  Let $W_\zeta$ denote $\psi_{\zeta,1}^{-1}(H)$.
Then each $W_\zeta$ is transverse to $\eta_\zeta$, and there is a $b>0$
such that the ball of radius $b$  centred at $a-b$ is tangent to $(H_\zeta^{j})^{-1}(W_\zeta)$
for all $\zeta$, where $H^j_\zeta$ is as in the previous section.  Letting
$F_\zeta$ be maps as in Theorem \ref{normexpose}, the maps
$\psi_{\zeta,1}\circ F_\zeta$ will satisfy the claims of the theorem.   Finally,
as before if $\overline\Omega$ is polynomially convex, the maps $\psi_{\zeta,1}$
may be taken to be holomorphic automorphisms.

\subsection{Proof of Theorem \ref{thm:exposing family} and Theorem \ref{thm:diff to ball}}
In this subsection,
we give the proofs of Theorem \ref{thm:exposing family} and Theorem \ref{thm:diff to ball},
which are based on Theorem \ref{thm:exposing family manifolds}.

\begin{proof}
(of Theorem \ref{thm:exposing family}) For $R>0$, let $B_R$ be the ball in $\mc^n$ centered at the origin with radius $R$,
and let $S_R=\partial B_R$ be the boundary of $B_R$.
We want to prove that there is a smooth family of curves $\gamma:\partial D\times[0,1]\rightarrow\partial D\times\mc^n$
which satisfies the conditions 1)-4) in Theorem \ref{thm:exposing family manifolds},
with $H$ replaced by $S_R$ for $R>>1$.

By assumption,
there is a smooth map $F:\partial D\times\mc^n\rightarrow\mc^n$
such that for each $\zeta\in\partial D$,
$F_\zeta(\cdot):=F(\zeta, \cdot):\mc^n\rightarrow\mc^n$ is a diffeomorphism
and $F_\zeta(D)$ is exposed at $F_\zeta(\zeta)$.

We assume that all $F_\zeta$ are orientation preserving.
For any $\zeta\in\partial D$ there is a natural isotopy
$H_\zeta:[0,1]\times \mc^n\rightarrow \mc^n$
from $F_\zeta$ to the identity,
given by $F_{\zeta,t}=F_\zeta(t, z):=\frac{1}{t}F_\zeta(tz)$ for $t\neq 0$,
and smoothly extended to $0$ by setting $F_{\zeta,0}(z)=z$.

This isotopy gives us a smooth family of vector fields $X_\zeta(t,z)$ on $\mc^n$.
In other words,
$F_{\zeta,t}$ is generated by $X_\zeta$.
It is clear that $X_\zeta(t,z)$ is even smooth jointly with respect to $\zeta, t,$ and $z$.

Let $r>0$ such that $F_{\zeta, t}(\overline D)\subset B_r(0):=\{z\in\mc^n:||z||<r\}$
for all $\zeta\in\partial D$ and $t\in [0,1]$.
Let $R>>r$ be a positive number.
Let $\psi$ be a positive smooth function on $\mc^n$ such that
$\psi\equiv 1$ on $B_{r}$ and $\psi\equiv 0$ on $\mc^n\backslash B_{R}$.
Let $X'_\zeta(t,z)=\psi X_\zeta(t,z)$,
then $X'$ has compact support.

Let $G_\zeta:\mc^n\rightarrow\mc^n$ be the diffeomorphism of $\mc^n$
generated by $X'_\zeta(t,z), t\in[0,1]$.
Then $G_\zeta=F_\zeta$ on $\overline D$ and $G_\zeta=Id$ on $\mc^n\backslash B_{R}$.

For $\zeta\in\partial D$, let $n_\zeta$ be the unit out-pointing normal vector of $\partial G_\zeta(D)$ at $G_\zeta(\zeta)$.
Then there is a unique $r_\zeta\in (0,+\infty)$ such that $G_\zeta(\zeta)+r_\zeta n_\zeta\in S_R$.
We define a smooth family of curves $\gamma:\partial D\times [0,1]\rightarrow\partial D\times\mc^n$
by setting $\gamma(\zeta, t)=(\zeta, G^{-1}_\zeta(G_\zeta(\zeta)+r_\zeta tn_\zeta))$.

Then the theorem follows by applying Theorem \ref{thm:exposing family manifolds} with $\gamma$ and $H=S_R$.
\end{proof}

\begin{proof}(of Theorem \ref{thm:diff to ball})
For $n=2$, we already assume that $\overline D$ is diffeomorphic to the closed unit ball.
For $n>2$, we want to show that $\overline D$ is diffeomorphic to the closed unit ball too.
By the Collar Theorem (see Theorem 6.1 in \cite{Hirsch}),
$\partial D$ is simply connected.
Let $z_0\in D$ and let $B$ be an open ball centered at $z_0$ with boundary $S\subset D$.
Let $W=\overline D\backslash B$,
then the triad $(W; S, \partial D)$ is an $h$-cobordism.
Note that the relative homology $H_*(W, S)=0$.
By the $h$-Cobordism Theorem (see Theorem 9.1 in \cite{Milnor65}),
$W$ is diffeomorphic to $S\times [0,1]$ and hence $\overline D$ is contractible.
So $\overline D$ is diffeomorphic the the unit ball (See Proposition A in \S 9 in \cite{Milnor65}).
By a result in differential topology (see Theorem 3.1 in \cite{Hirsch}),
there is a diffeomorphism $\sigma$ of $\mc^n$
such that $\sigma(\overline D)$ is the closed unit ball.
So the Theorem follows from Theorem \ref{thm:exposing family}.
\end{proof}

\maketitle

\end{document}